\documentclass{amsart}
\newtheorem{theorem}{Theorem}[section]
\newtheorem{lemma}[theorem]{Lemma}
\newtheorem{proposition}[theorem]{Proposition}

\theoremstyle{definition}

\newtheorem{example}[theorem]{Example}
\newtheorem{problem}[theorem]{Probelm}
\theoremstyle{remark}
\newtheorem{remark}[theorem]{Remark}
\numberwithin{equation}{section}

\begin{document}
\title{A Mazur--Ulam theorem in non-Archimedean normed spaces}

\author[M.S. Moslehian, Gh. Sadeghi]{Mohammad Sal Moslehian and Ghadir Sadeghi}
\address{Department of Mathematics, Ferdowsi University, P. O. Box
1159, Mashhad 91775, Iran; \newline Centre of Excellence in Analysis
on Algebraic Structures (CEAAS), Ferdowsi University, Iran.}
\email{moslehian@ferdowsi.um.ac.ir and moslehian@ams.org}
\address{Department of Mathematics, Ferdowsi
University, P. O. Box 1159, Mashhad 91775, Iran;
\newline Banach Mathematical Research Group (BMRG), Mashhad, Iran.}
\email{ghadir54@yahoo.com} 


\begin{abstract}
The classical Mazur--Ulam theorem which states that every surjective
isometry between real normed spaces is affine is not valid for
non-Archimedean normed spaces. In this paper, we establish a
Mazur--Ulam theorem in the non-Archimedean strictly convex normed
spaces.
\end{abstract} \maketitle

\section{Introduction and preliminaries}

A \emph{non-Archimedean field} is a field $\mathcal{K}$ equipped
with a function (valuation) $|\cdot|$ from $\mathcal{K}$ into $[0,
\infty)$ such that $|r|=0$ if and only if $r=0$, $|rs|=|r|\,|s|$,
and $|r+s|\leq \max\{|r|, |s|\}$ for all $r, s \in \mathcal{K}$.
Clearly $|1|=|-1|=1$ and $|n| \leq 1$ for all $n \in {\mathbb N}$.
An example of a non-Archimedean valuation is the mapping $|\cdot|$
taking everything but $0$ into $1$ and $|0|=0$. This valuation is
called trivial.

\noindent In 1897, Hensel \cite{HEN} discovered the $p$-adic numbers
as a number theoretical analogue of power series in complex
analysis. Fix a prime number $p$. For any nonzero rational number
$x$, there exists a unique integer $n_x \in {\mathbb Z}$ such that
$x= \frac{a}{b} p^{n_x}$, where $a$ and $b$ are integers not
divisible by $p$. Then $|x|_p : = p^{-n_x}$ defines a
non-Archimedean norm on ${\mathbb Q}$. The completion of ${\mathbb
Q}$ with respect to the metric $d(x, y) = |x - y|_p$ is denoted by
${\mathbb Q}_p$ which is called the \emph{$p$-adic number field};
see \cite{VAN}. During the last three decades $p$-adic numbers have
gained the interest of physicists for their research in particular
in problems coming from quantum physics, $p$-adic strings and
superstrings (cf. \cite{KHR}).

Let ${\mathcal X}$ be a vector space over a scalar field
$\mathcal{K}$ with a non-Archimedean valuation $|\cdot|$. A function
$\| \cdot \| : {\mathcal X} \to [0, \infty)$ is said to be a
\emph{non-Archimedean norm} if it satisfies the following
conditions:

(i) $\|x\| = 0$ if and only if $x=0$;

(ii) $\|rx\| = |r| \|x\| \quad (r \in {\mathcal K}, x \in {\mathcal
X})$;

(iii) the strong triangle inequality
$$\| x+ y\| \leq \max \{\|x\|, \|y\|\} \quad (x, y \in {\mathcal X}).$$
Then $({\mathcal X}, \|\cdot\|)$ is called a non-Archimedean normed
space. A non-Archimedean normed space over a valued filed
$\mathcal{K}$ with $|2|=1$ satisfying $\|\mathcal{X}\|:=\{\|x\| : x
\in \mathcal{X}\}=\{|r| : r \in \mathcal{K}\}$ is called
\emph{strictly convex} if $\|x+y\|=\max\{\|x\|, \|y\|\}$ and
$\|x\|=\|y\|$ imply $x=y$. The assumption $|2|=1$ is necessary in
order a space ${\mathcal X}$ to be satisfies the implication. A
non-Archimedean normed space is called \emph{spherically complete}
if every collection of closed balls in ${\mathcal X}$ which is
totally ordered by inclusion has a non-empty intersection
\cite{VAN}. Every spherically complete space is complete but the
converse is not true in general. The notion of spherical
completeness is more suitable than the notion of completeness in the
study of non-Archimedean spaces. Theory of non-Archimedean normed
spaces is not trivial, for instance there may not be any unit
vector. Although many results in classical normed space theory have
a non-Archimedean counterpart, but their proofs are essentially
different and require an entirely new kind of intuition, cf.
\cite{M-R, N-B}.

A valued space is a non-Archimedean normed linear space over a
trivially field with characteristic $0$. A $V$-space $\mathcal{X}$
is a valued space which is complete in its norm topology such that
$$\|\mathcal{X}\|\subset \{0\}\cup
\{\rho^{n} :n \in \mathbb{Z}\}$$ for some real number $\rho >1$
\cite{ROB}.

The theory of isometric mappings had its beginning in the classical
paper \cite{M-U} by S. Mazur and S. Ulam, who proved that every
isometry of a real normed  vector space onto another real normed
vector space is a linear mapping up to translation. The property is
not true for normed complex vector spaces (for instance, consider
the conjugation on $\mathbb{C}$). The hypothesis surjectivity is
essential. Without this assumption, J. A. Baker \cite{BAK} proved
that every isometry from a normed real space into a strictly convex
normed real space is a linear up to translation. A number of
mathematicians have had deal with Mazur--Ulam theorem; see
\cite{R-S, R-X} and references therein. Mazur--Ulam Theorem is not
valid in the contents of non-Archimedean normed spaces, in general.
As a counterexample take $\mathbb{R}$ with the trivial
non-Archimedean valuation and define $f: \mathbb{R} \to \mathbb{R}$
by $f(x)=x^3$. Then $f$ is clearly a surjective isometry and
$f(0)=0$, but $f$ is not linear.

In this paper, by using some ideas of \cite{BAK}, we establish a
Mazur--Ulam type theorem in the framework of strictly convex
non-Archimedean normed spaces. We also provide an example to show
that the assumption of strict convexity is essential.


\section{Main results}

\begin{lemma}\label{L}
Let $\mathcal{X}$ be a non-Archimedean normed space over a valued
filed $\mathcal{K}$ which is strictly convex and let $x, y \in
\mathcal{X}$. Then $\frac{x+y}{2}$ is the unique member of
$\mathcal{X}$ which is of distance $\|x-y\|$ from both $x$ and $y$.
\end{lemma}
\begin{proof}
There is nothing to prove if $x=y.$ Let $x\neq y.$ The point
$\frac{x+y}{2}$ is of distance $\|x-y\|$ from both $x$ and $y$,
since
$$\|x-\frac{x+y}{2}\|=\|\frac{x-y}{2}\|=\|x-y\|\,;$$
$$\|y-\frac{x+y}{2}\|=\|\frac{x-y}{2}\|=\|x-y\|\,.$$
Assume that $z, t \in \mathcal{X}$ with
\begin{eqnarray*}
\|x-z\|=\|x-t\|=\|y-z\|=\|y-t\|=\|x-y\|\,.
\end{eqnarray*}
Then
\begin{eqnarray}\label{equ1}
\|x-\frac{z+t}{2}\|\leq \max\{\|\frac{x-z}{2}\|,
\|\frac{x-t}{2}\|\}=\|x-y\|\,.
\end{eqnarray}
 Similarly
\begin{eqnarray}\label{equ2}
\|y-\frac{z+t}{2}\|\leq \|x-y\|
\end{eqnarray}
If both of inequalities (\ref{equ1}) and (\ref{equ2}) were strict we
would have
\begin{eqnarray*}
\|x-y\|\leq \max\{\|x-\frac{z+t}{2}\|, \|y-\frac{z+t}{2}\|\} <
\|x-y\|\,,
\end{eqnarray*}
a contradiction. So at least one of the equalities holds in
(\ref{equ1}) and (\ref{equ2}). Without lose of generality assume
that equality holds in (\ref{equ1}). Then $\|x-\frac{z+t}{2}\|=
\max\{\|\frac{x-z}{2}\|, \|\frac{x-t}{2}\|\}.$ By strictly convexity
we obtain $\frac{x-z}{2}=\frac{x-t}{2}$ that is $z=t.$
\end{proof}

\begin{theorem} \label{main} Suppose that $\mathcal{X}$ and $\mathcal{Y}$ are non-Archimedean normed spaces and
$\mathcal{Y}$ is strictly convex. If $f:\mathcal{X} \to \mathcal{Y}$
is an isometry, then $f-f(0)$ is additive.
\end{theorem}

\begin{proof}
Let $g(x)=f(x)-f(0).$ Then $g$ is an isometry and $g(0)=0.$

\begin{eqnarray*}
\|g\left(\frac{x+y}{2}\right)-g(x)\|=\|\frac{x+y}{2}-x\|=\|x-y\|=\|g(x)-g(y)\|
\quad (x, y \in \mathcal{X})
\end{eqnarray*}
and similarly
\begin{eqnarray*}
\|g(\frac{x+y}{2})-g(y)\|=\|g(x)-g(y)\| \quad (x, y \in \mathcal{X})
\end{eqnarray*}

It follows from Lemma \ref{L} that
$$g(\frac{x+y}{2})=\frac{g(x)+g(y)}{2}$$
Hence $g=f-f(0)$ is additive.
\end{proof}
\begin{remark}
If $\mathcal{K}=\mathbb{Q}_{p}$ with the valuation $|.|_p$, where
$p\neq 2$ then $f-f(0)$ in Theorem \ref{main} is a linear mapping.
\end{remark}

\begin{example}
Let $\mathcal{Y}$ be a $V$-space over a trivially valued field
$\mathcal{K}$ such that $1\in \|\mathcal{Y}\|$ i.e there exists
$y_{0}\in \mathcal{Y}$ with $\|y_{0}\|=1.$  Clearly $|2|=1$ and
$\mathcal{Y}$ is not strictly convex. Now define the mapping
$f:\mathcal{K} \to \mathcal{Y}$ by $f(\alpha)=\alpha^{3}y_{0}.$ Then
$f$ is an isometry and $f(0)=0$ but $f$ is not additive. Therefore
the assumption that $\mathcal{Y}$ is strictly convex cannot be
omitted in Theorem \ref{main}.
\end{example}

It is well know that $\mathbb{Q}_{2}$ is spherically complete (see
\cite{SCH}) and $|2|_2 \neq 1.$ Let $f:\mathbb{Q}_{2} \to
\mathbb{Q}_{2}$ be defined by
$$f(x)=\left \{ \begin{array}{ll}
1/x &  x=\frac{a}{b}2^{0} \textrm{~for ~some~} a\neq 0, b \textrm{~with~} (a, b)=1 \\
x&\textrm{otherwise}\\
\end{array}\right.$$
Then $f$ is an isometry (note that in a non-Archimedean field
$|a-b|=\max\{|a|, |b|\}$ for all $a, b$ with $|a| \neq |b|$),
$f(0)=0$ but $f$ is not an additive mapping. Regarding this fact and
using some ideas of \cite{VAI}, we have the next result which seems
to be interesting on its own right.

\begin{proposition} Suppose that $\mathcal{X}$ and
$\mathcal{Y}$ are non-Archimedean normed spaces on a non-Archimedean
field $\mathcal{K}$ with $|k|\neq 1$ for some $k \in {\mathbb N}$.
Assume that $\mathcal{X}$ or $\mathcal{Y}$ is spherically complete.
If $f:\mathcal{X} \to \mathcal{Y}$ is a surjective isometry, then
for each $u \in \mathcal{X}$ there exists a unique $v \in
\mathcal{X}$ such that $f(u)+f(v)=f\left(\frac{u+v}{k}\right).$
\end{proposition}
\begin{proof}
We first show that both spaces are spherically complete. Let
$\mathcal{X}$ be a spherically complete and $\mathcal{B}$ is a
collection of closed balls in $\mathcal{Y}$ which is totally
ordered by inclusion. Since $f$ is a surjective isometry,
$f^{-1}(\mathcal{B})$ is a collection of closed balls in
$\mathcal{X}$ which is totally ordered by inclusion. Thus $\cap
f^{-1}(\mathcal{B}) \neq \emptyset $. Therefore we have $\cap
\mathcal{B}\neq \emptyset.$ Similarly one can prove that if
$\mathcal{Y}$ is spherically complete then so is $\mathcal{X}$.

Let $u \in \mathcal{X}$. The mapping $\varphi:\mathcal{X} \to
\mathcal{X}$ defined by $\varphi(x):=kx-u$ is a contractive mapping
since
\begin{eqnarray*}
\|\varphi(x)-\varphi(y)\|=\|kx-ky\|=|k| \|x-y\|<\|x-y\|.
\end{eqnarray*}
Define the isometry mapping $\psi:\mathcal{Y}\to \mathcal{Y}$ by
$\psi(y):=f(u)+y.$

\noindent The mapping $h:=\varphi f^{-1}\psi f$ is a contractive
mapping on $\mathcal{X}$, since
\begin{eqnarray}\label{eq0}
\|h(x)-h(y)\|&=&\|\left(\varphi f^{-1}\psi f\right)(x)-\left(\varphi
f^{-1}\psi f\right)(y)\|\nonumber\\&<& \|(f^{-1}\psi
f)(x)-(f^{-1}\psi
f)(y)\|\nonumber\\
&=&\|(\psi f)(x)-(\psi f)(y)\|\nonumber\\
&=& \|f(x)-f(y)\|\nonumber\\
&=&\|x-y\|.
\end{eqnarray}
By \cite[Theorem 1]{P-V} $h$ has a unique fixed point $v$. Then we
have
\begin{eqnarray*}
\left(\varphi f^{-1}\psi
f\right)(v)=h(v)=v=\varphi\left(\frac{u+v}{k}\right)
\end{eqnarray*}
Therefore $\psi(f(v))=f\left(\frac{u+v}{k}\right)$, since $\varphi$
is one to one. Hence $f(u)+f(v)=f\left(\frac{u+v}{k}\right)$.
\end{proof}

Finally, suppose that $E$ and $F$ are metric spaces and $f : E\to F$
is a mapping. A real number $r>0$ is called a conservative distance
for $f$ if $d(x,y)=r$ implies $d(f(x),f(y))=r$. In 1970, A. D.
Aleksandrov \cite{ALE} posed the following problem: "Under what
conditions is a mapping preserving a distance $r$ an isometry?" This
problem is not easy to solve even in the case where $E$ and $F$ are
normed spaces. A number of mathematicians have discussed Aleksandrov
problem under certain additional conditions, see \cite{R-S, R-X}. In
the spirit of the Mazur--Ulam theorem, we pose the following
problem:
\begin{problem}[\textbf{Aleksandrov Problem in non-Archimedean spaces}]
Let $\mathcal{X}$ and $\mathcal{Y}$ be non-Archimedean normed
spaces. Under what conditions is a mapping preserving a distance $r$
an isometry?
\end{problem}
Another related subject is study of the stability of isometries (see
\cite{SEM}) in the framework of non-Archimedean normed spaces as
follows.
\begin{problem}[\textbf{Stability of isometries in non-Archimedean spaces}]
Let $\mathcal{X}$ and $\mathcal{Y}$ be non-Archimedean normed spaces
and $f: {\mathcal X} \to {\mathcal Y}$ be a mapping satisfying
$$|\,\|f(x)-f(y)\| - \|x-y\|\,| \leq \varepsilon$$
for some $\varepsilon$ and for all $x, y \in {\mathcal X}$. Under
what assumptions are there a constant $\kappa$ and an isometry $T:
{\mathcal X} \to {\mathcal Y}$ such that $\|f(x)-T(x)\| \leq \kappa
\varepsilon$ for all $x \in {\mathcal X}$?
\end{problem}

\bibliographystyle{amsplain}

\end{document}